\documentclass[12pt]{article}

\usepackage{latexsym, graphicx}
\usepackage{amsmath, amssymb, amsthm}
\usepackage{enumerate}
\usepackage{url}
\usepackage[margin=1.5cm]{caption}

\newtheorem{theorem}{Theorem}[section]

\newtheorem{lem}[theorem]{Lemma}

\newtheorem{conjecture}[theorem]{Conjecture}

\theoremstyle{definition}

\DeclareMathOperator{\mad}{\rm{mad}}

\usepackage{fullpage}

\makeatletter \newcommand{\listoftodos}{\section*{Todo list} \@starttoc{tdo}}
\newcommand\l@todo[2]
    {\par\noindent \textit{#2}, \parbox{10cm}{#1}\par} \makeatother

\begin{document}

\title{Choosability of the square of a planar graph\\ with maximum degree four }

\author{
Daniel W. Cranston\thanks{Virginia Commonwealth University, Department of
Mathematics and Applied Mathematics, Richmond, VA, USA. Email: {\tt
dcranston@vcu.edu}}
\and
Rok Erman\thanks{Institute of Mathematics, Physics and Mechanics, 
Jadranska 19, 1000 Ljubljana, Slovenia. 
Email: {\tt rok.erman@gmail.com}
} \and
Riste \v Skrekovski\thanks{Department of Mathematics, University of Ljubljana, 
Jadranska 21, 1000 Ljubljana, Slovenia. 
Email: {\tt skrekovski@gmail.com}
}
}

\date{\today}

\maketitle

\noindent
{\bf Keywords:} choosability; square of a graph;
maximum average degree; discharging; girth; maximum degree; list-colouring \\


\begin{abstract}
We study squares of planar graphs with the aim to determine their list
chromatic number.  We present new upper bounds for the square of a
planar graph with maximum degree $\Delta \leq 4$. In particular $G^2$ is 5-,
6-, 7-, 8-, 12-, 14-choosable if the girth of $G$ is at least 16, 11,
9, 7, 5, 3 respectively.
In fact we prove more general results, in terms of maximum average degree, that
imply the results above.
\end{abstract}

\section{Introduction}

The square of a graph $G$, denoted by $G^2$, is the graph with $V(G^2)=V(G)$ and 
$E(G^2)=\{uv\mid  d_G(u,v)\leq 2\}$. This means that two vertices are adjacent 
in $G^2$ if they are at distance at most two in $G$.
If $\Delta$ is the maximum degree of $G$, then to colour its square $G^2$ we
will need at least $\Delta+1$ colours while the upper bound is $\Delta^2+1$
using the greedy algorithm.  This upper bound is also achieved for a few
graphs, for example by the Petersen graph.

Regarding the colouring of the square of planar graphs, Wegner~\cite{wegner}
posed the following conjecture in 1977:

\begin{conjecture}[Wegner]
\label{wegner}
For a planar graph $G$ of maximum degree $\Delta$:
$$\chi(G^2)\leq \left\{  \begin{array}{ll}
         7, & \mbox{$\Delta=3$};\\
         \Delta+5, & \mbox{$4\leq \Delta \leq 7$};\\
         \lceil\frac{3}{2}\Delta \rceil+1, & \mbox{$\Delta \geq 8$}.
         \end{array} \right.$$
\end{conjecture}

In~\cite{Havet1} Havet, van den Heuvel, McDiarmid, and Reed showed that 
the following holds: $\chi(G^2)\leq \frac{3}{2}\Delta(1+o(1))$,
which is also true for the choice number (defined below).
Dvo\v r\'ak, Kr\'al', Nejedl\'y, and \v Skrekovski~\cite{skreko1} showed that
the square of every planar graph of girth at least six with sufficiently large
maximum degree $\Delta$ is $(\Delta+2)$-colorable.
Borodin and Ivanova~\cite{borodin2} strengthened this result to prove 
that for every planar graph $G$ of girth at least six with maximum degree
$\Delta \ge 24$, the choice number of $G^2$ is at most $\Delta+2$.
For colouring (rather than list-colouring), the same authors 
showed~\cite{borodin3} that for every planar graph $G$ of girth at least six 
with maximum degree
$\Delta \ge 18$, the chromatic number of $G^2$ is at most $\Delta+2$.

Lih, Wang, and Zhu~\cite{Lih} showed that the square of a $K_4$-minor free
graph with maximum degree $\Delta$ has chromatic number at most $\lfloor
\frac{3}{2}\Delta \rfloor+1$ if $\Delta \geq 4$ and $\Delta+3$ if $\Delta
=2,3$.  The same bounds were shown to hold for the choice number by
Hetherington and Woodall~\cite{Hetherington}.

All graphs in this paper are undirected, simple, and finite. For standard graph
definitions see~\cite{diestel}.
Denote by $l(f)$ the length of a face $f$ and by $d(v)$ the degree of a vertex
$v$.  A {\it $k$-vertex} is a vertex of degree $k$.
A {\it $k^-$-vertex} is a vertex of degree at most $k$, and a {\it
$k^+$-vertex} is a vertex of degree at least $k$.  If a vertex $u$ is adjacent to
a $k$-vertex $v$, then $v$ is a $k$-neighbor of $u$.
A {\it thread} between two vertices with degree at least three is a path between
them consisting of only 2-vertices.  A {\it $k$-thread} is a thread with $k$
internal 2-vertices.
If vertices $u$ and $v$ lie on a common thread, then $u$ and $v$ are {\it
weak neighbors} of each other.  Similarly, we define a weak $k$-neighbor.

A {\it colouring} of the vertices of a graph $G$ is a mapping
$c:V(G)\rightarrow \mathbb{N}$; we call elements of $\mathbb{N}$ {\it colours}.
 A colouring is {\it proper} if every two adjacent vertices are mapped to
different colours.  List colouring was first studied by Vizing~\cite{vizing}
and is defined as follows.  Let $G$ be a simple graph. A {\it list-assignment}
$L$ 
is an assignment of lists of colours to vertices.  A {\it list-colouring} is
then a colouring where each vertex $v$ receives a colour from $L(v)$. The graph
$G$ is {\it $L$-choosable} if there is a proper $L$-list-colouring.
If $G$ has a list-colouring for every list-assignment with $|L(v)|=k$ for each
vertex $v$, then $G$ is {\it $k$-choosable}.  We will denote the size of
the lists of colours in a specific case simply by $\chi_l$.
The minimum $k$ such that $G$ is $k$-choosable is called the {\it choice number}
of $G$.

To prove our theorem we will use the discharging method, which was first used
by Wernicke~\cite{wernicke}; this technique is used to
prove statements in structural graph theory, and it is commonly applied in the
context of planar graphs. It is most well-known for its central role in the
proof of the Four Colour Theorem.  Here we apply the discharging method in the
more general context of the {\it maximum average degree}, denoted $\mad(G)$,
which is defined as $\mad(G):=\max_{H\subseteq G}\frac{2|E(H)|}{|V(H)|}$,
where $H$ ranges over all subgraphs of $G$.
A straightforward consequence of Euler's Formula is that every planar graph $G$
with girth at least $g$ satisfies $\mad(G)<\frac{2g}{g-2} = 2+\frac{4}{g-2}$.
We call this Fact~1.  Most of our results for planar graphs will follow from
corresponding results for maximum average degree, via Fact 1.

The key tool in many of our proofs is global discharging, which relies on
reducible configurations that may be arbitrarily large.  Global discharging was
introduced by Borodin~\cite{borodin1}.  Typically, the vertices in these
reducible configurations have degrees only 2 and $\Delta$.  Our innovation in
this paper is that we consider arbitrarily large reducible configurations
consisting entirely of 2-vertices and 3-vertices, even though $\Delta=4$.  For
two similar applications of global discharging, see~\cite{cranston1}
and~\cite{borodin2}.

Kostochka and Woodall~\cite{kostochka} conjectured that, for every square of a
graph, the chromatic number and choice number are the same:
\begin{conjecture}[Kostochka and Woodall]
\label{kostochka}
Let $G$ be a simple graph. Then $$\chi_l(G^2) =\chi(G^2).$$
\end{conjecture}

When $G$ is a planar graph,
the upper bound on $\chi(G^2)$ in terms of $\Delta$ was succesively improved by Jonas~\cite{Jonas},
 Wong~\cite{Wong},
Van den Heuvel and McGuinness~\cite{Heuvel},
 Agnarsson and Halld´orsson~\cite{Agnarsson},
  Borodin et al.~\cite{borodin} and finally by
 Molloy and Salavatipour~\cite{Molloy} to the best known upper bound so far:
$\chi(G^2)\leq \lceil \frac{5}{3}\Delta\rceil+78.$

The choosability of squares of subcubic planar graphs has been extensively
studied by Dvo\v r\'ak, \v Skrekovski, and Tancer~\cite{skreko}, Montassier and
Raspaud~\cite{montassier}, Thomassen~\cite{thomassen}, Havet~\cite{havet}, and
Cranston and Kim~\cite{cranston}.
For the case $\Delta=4$ there have been no results so far.
We give some upper bounds on $\chi_l(G^2)$ when $\Delta(G)= 4$ and $\mad(G)$
is bounded.  These results imply bounds for $\chi_l(G^2)$ when $G$ is planar
with prescribed girth:

\begin{theorem}
\label{mainthm}
Let $G$ be a graph with maximum degree $\Delta=4$.
The following bounds hold:
\begin{enumerate}[(a)]
	\item[$(a)$]  $G^2$ is $5$-choosable if $\mad(G)<16/7$,
specifically, if $G$ is planar with girth at least 16.
	\item[$(b)$]  $G^2$ is $6$-choosable if $\mad(G)<22/9$,
specifically, if $G$ is planar with girth at least 11.
	\item[$(c)$]  $G^2$ is $7$-choosable if $\mad(G)<18/7$,
specifically, if $G$ is planar with girth at least 9.
	\item[$(d)$]  $G^2$ is $8$-choosable if $\mad(G)<14/5$,
specifically, if $G$ is planar with girth at least 7.
	\item[$(e)$]  $G^2$ is $12$-choosable if $\mad(G)<10/3$,
specifically, if $G$ is planar with girth at least 5.
	\item[$(f)$]  $G^2$ is $14$-choosable if $G$ is planar.
\end{enumerate}
\end{theorem}

This theorem is summarized in the following table:
\begin{table}[!h]%
\begin{center}
		\begin{tabular}{ r |c |c |c |c |c |c }
	
  $\chi_l\leq$ & $5$ & $6$ & $7$ & $8$ & $12$ & $14$ \\
  \hline
  $\mad(G)<$ & $16/7$ & $22/9$ & $18/7$ & $14/5$ & $10/3$ & $-$ \\
  \hline
  $\mbox{planar and~}g\geq$ &  $16$ & $11$ & $9$ & $7$ & $5$ & $3$  \\

\end{tabular}
\caption{Upper bounds on the choice number for squares of graphs with $\Delta=4$
and bounded maximum average degree, including planar graphs with bounded girth.}
\label{}
\end{center}
\end{table}
\vspace{-.2in}

We will prove each of the claims by contradiction while studying the smallest counterexample to the claim with respect to the number of vertices.
If we remove one or more vertices from this graph we know that its square can be properly coloured with the lists provided.
We will use this fact in the proofs of the claims.
\subsection{Reducible configurations}

A {\it configuration} is an induced subgraph $C$ of a graph $G$. We call a
configuration {\it reducible} if it cannot appear in a minimal counterexample.
To prove that a configuration is reducible, we infer from the minimality of
$G$ that subgraph $G-H$ can be properly coloured, and then prove that this
colouring can be extended to a proper colouring of the original graph $G$;
this gives a contradiction.
A configuration is {\it $k$-reducible} if it is reducible in the setting of
$k$-choosability. Clearly a $k$-reducible configuration is also
$(k+1)$-reducible.

We split our proof of the main theorem into six lemmas, one for each part of
the theorem.  Within each lemma, we prove the reducibility of the
configurations used in that lemma.
Once we prove a configuration is reducible, we will assume that such
a configuration is not present in a minimal counterexample to that lemma.

We will prove that the configurations are reducible by using the same method
each time: remove some vertices and colour the remaining graph by minimality.
If necessary uncolour some vertices, and finally extend this colouring to the
whole graph.

To simplify the presentation of the reducibility proofs we give figures using
the following notation: A removed vertex is marked with a square around it.
An uncoloured vertex is marked with a circle around it.  The minimal number of
colours left in the list of a removed or uncoloured vertex is written next to
it.
These figures allow the reader to quickly verify that the configurations
pictured are reducible.
In the first few reducibility proofs we will provide detailed reasoning but in
the remaining ones we will only present the corresponding figure and
leave the details to the reader.

We call a graph {\it degree choosable} if it can be colored from any list
assignment $L$ such that $|L(v)|=d(v)$ for all $v\in V(G)$.
For a few of the reducibility proofs, we will need the following result of
Erd\H{o}s, Rubin, and Taylor~\cite{ERT}:

\begin{lem}[Choosability Lemma]
\label{ERTlemma}
A connected graph fails to be degree-choosable if and only if every block is a
complete graph or an odd cycle.
\end{lem}

\section{Proof of the Main Theorem}

In this section, we prove our main result, Theorem~\ref{mainthm}.  The six
parts of Theorem~\ref{mainthm} are completely independent, so we present the
proof as six self-contained lemmas, each proving a corresponding part
of the theorem.  The proofs of Lemmas~\ref{lemma1}--\ref{lemma5} all use
maximum average degree, while Lemma~\ref{lemma6} requires planarity, since it
sends charge to faces.  The proofs of Lemmas~\ref{lemma1}, \ref{lemma2}, and
\ref{lemma4} make use of global discharging; the easiest of these proofs is
Lemma~\ref{lemma1}, while Lemmas~\ref{lemma2} and \ref{lemma4} require
additional details and subtlety.  We now prove the six lemmas without further
comment.

\begin{lem}
\label{lemma1}
If $\Delta(G)\le 4$ and $\mad(G)<16/7$, then $\chi_l(G^2)\le 5$.  In particular,
for every planar graph $G$ with $\Delta(G)\le 4$ and girth at least 16, we have
$\chi_l(G^2)\le 5$.
\end{lem}
\begin{proof}
The second statement follows from the first by Fact~1.  To prove the first, we
use discharging.  Let $G$ be a minimal counterexample to the lemma, i.e., a
minimal graph with $\Delta(G)\le 4$ and $\mad(G)<16/7$ such that
$\chi_l(G^2)>5$.
For each vertex $v$, we begin with charge $\mu(v)=d(v)$; we will show that
after the discharging phase each vertex finishes with charge at least $16/7$,
which gives a contradiction and proves the lemma.

We call a configuration {\it 5-reducible} if it cannot appear in a minimal
counterexample to the lemma. We use the following configurations (see Fig.~\ref{slika1}):
\begin{itemize}
    \item[$(i)$] \textit{A 4-thread is 5-reducible.}
        Let $v$ and $w$ be the middle two vertices of 
        the 4-thread.  By the minimality of $G$ we can 5-list-color 
        $(G\setminus\{v,w\})^2$.  Now $v$ and $w$ each have at least 2 colors 
        available, so we can extend the coloring to $G$.
    \item[$(ii)$] \textit{ A 3-thread $S$ incident to a 3-vertex $u$ is 5-reducible.} 
        Let $v$ be the 2-vertex on $S$ adjacent to $u$ and let $w$ be the 2-vertex
        adjacent to $v$.  By the minimality of $G$, we can 5-list-color 
        $(G\setminus\{v,w\})^2$.  Now $u$ and $v$ have at least one and two available
        colors, respectively.  So we can extend the coloring to $G$ by coloring $u$
        then $v$.
    \item[$(iii)$] \textit{A $3k$-cycle $C_{3k}$ with $d(v_{3i})=3$ for all $i$ and
        $d(v_{3i+1})=d(v_{3i+2})=2$ for all $i$ (the subscripts are modulo $3k$) 
        is 5-reducible.}  Let $S=\{v_{3i}:1\le i\le k\}$.  We delete all 
        vertices on $C_{3k}$ with degree 2. Now the subgraph of $G^2$ that we 
        must color, $(C_{3k})^2\setminus S\cong C_{2k}$, is isomorphic to an even cycle.
        Each uncolored vertex has at most 3 restrictions on its color, so it 
        has a list of at least 2 available colors. Now we can extend the coloring to $G$ since $\chi_l(C_{2k})=2$ (this is an easy exercise, and also follows immediately from the Choosability Lemma).
\end{itemize}

Let $H$ denote the subgraph of $G$ induced by 2-threads with 3-vertices at both
ends.  Since configuration $(iii)$ is 5-reducible, $H$ must be acyclic.
Since every tree has one more vertex than edge, we can recursively assign each
2-thread in $H$ to be sponsored by an incident 3-vertex such that each 3-vertex
sponsors at most one 2-thread.  

\begin{figure}[htp!]
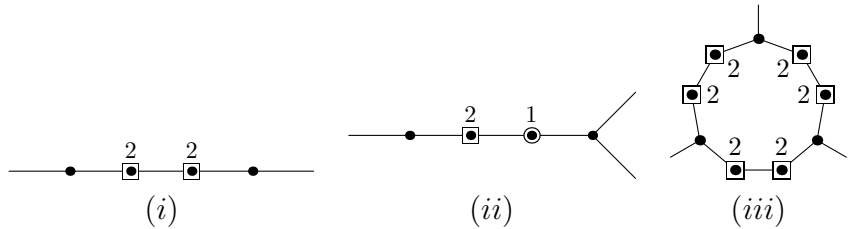

	\centerline{
		\begin{tabular}{ccc}
            \includegraphics[scale=0.9]{fig11.1}\hspace{.5cm}&
            \includegraphics[scale=0.9]{fig11.2}\hspace{.5cm}&
            \includegraphics[scale=1]{fig11.3}\\
            $(i)$ & $(ii)$ & $(iii)$
		\end{tabular}}
	\caption{Configurations $(i)$, $(ii)$, and $(iii)$ from Lemma~\ref{lemma1} are 5-reducible.}
\label{slika1}
\end{figure}

We use the initial charge function $\mu(v)=d(v)$ and the following discharging
rules.

\begin{itemize}
\item {\bf R1:} Each $3$-vertex gives charge $1/7$ to each incident thread.
\item {\bf R2:} Each 4-vertex gives charge $3/7$ to each incident thread.\footnote{If a 4-vertex $v$ is adjacent to two vertices in the same thread,
i.e., $v$ serves as both endpoints of the thread, then $v$ sends twice the
normal charge to the thread; similarly for Lemma~2.}
\item {\bf R3:} Each $3$-vertex incident with a 2-thread  that it sponsors
gives an additional charge of $2/7$ to that 2-thread.
\end{itemize}

Now we show that each $3^+$-vertex finishes with charge at least $16/7$ and
that each $k$-thread receives charge at least $2k/7$ (so that it finishes with
charge at least $16k/7$).
Note that a 1-vertex is 5-reducible, so $\delta(G)\ge 2$.
First we consider $3^+$-vertices.
If $d(v)=4$, then $v$ gives charge $3/7$ to each incident thread, so
$\mu^*(v)\ge 4-4(3/7)=16/7$.
If $d(v)=3$, then $v$ sends charge $1/7$ to each incident thread and an
additional charge of $2/7$ to at most one incident thread, so $\mu^*(v)\ge 3 -
3(1/7)-1(2/7)=16/7$.

Now we consider threads.
Each 3-thread receives charge $3/7$ from each endpoint, which are both
4-vertices by $(ii)$.  Each 1-thread receives charge at least $1/7$ from each
endpoint.  Each 2-thread with at least one degree 4 endpoint receives charge
$3/7$ from one endpoint and at least $1/7$ from the other.  Finally, each
2-thread with two degree 3 endpoints receives charge $1/7$ from each endpoint
and an additional charge of $2/7$ from its sponsor, for a total of $4/7$.
Thus $\mad(G)\ge 16/7$.  This contradiction completes the proof.
\end{proof}

\begin{lem}
\label{lemma2}
If $\Delta(G)\le 4$ and $\mad(G)<22/9$, then $\chi_l(G^2)\le 6$.  In particular,
for every planar graph $G$ with $\Delta(G)\le 4$ and girth at least 11, we have
$\chi_l(G^2)\le 6$.
\end{lem}
\begin{proof}
The second statement follows from the first by Fact~1.  To prove the first,
we use discharging.  Let $G$ be a minimal counterexample to the lemma.
For each vertex $v$, we begin with charge $\mu(v)=d(v)$,
and we will show that after discharging each vertex finishes with charge at
least $22/9$, which gives a contradiction and proves the lemma.

\begin{figure}[htp!]
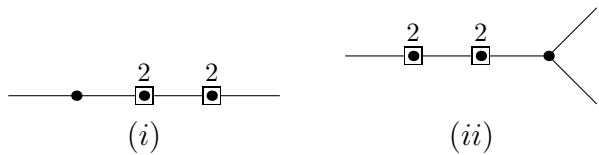

	\centerline{
		\begin{tabular}{ccc}
            \includegraphics[scale=1]{fig21.1}&\hspace{2cm}&\includegraphics[scale=1]{fig21.2}\\
            $(i)$ & &$(ii)$
		\end{tabular}}
\caption{Configurations $(i)$ and $(ii)$ from Lemma~\ref{lemma2} are 6-reducible. }
\label{slika11}
\end{figure}

We call a configuration {\it 6-reducible} if it cannot appear in a minimal
counterexample to the lemma. We use the following configurations (see Fig.~\ref{slika11}):
\begin{itemize}
    \item[$(i)$]\textit{ A 3-thread $S$ is 6-reducible.} 
        Let $v$ and $w$ be adjacent 2-vertices on $S$, with $v$ adjacent to an
        endpoint of $S$.  By the minimality of $G$ we can 6-list-color
        $(G\setminus\{v,w\})^2$.  Now $v$ and $w$ have at least 1 and 3 colors
        available, respectively.  So we can extend the coloring to $G$ by coloring $v$
        then $w$.
    \item[$(ii)$] \textit{A 2-thread $T$ incident to a 3-vertex $u$ is 6-reducible.}
        Let $v$ and $w$ be the two 2-vertices of $T$, with $v$ adjacent to $u$.  By the
        minimality of $G$ we can 6-list-color $(G\setminus\{v,w\})^2$.  Now $v$ and $w$
        have at least 2 and 1 colors available, respectively.  So we can extend the
        coloring to $G$ by coloring $w$ then $v$.
\end{itemize}
Let $H$ be the subgraph induced by 2-threads; recall that the endpoints of each
2-thread must be 4-vertices, by $(ii)$.
As in the proof of Lemma~1, $H$ must be acyclic.  Thus, we can assign each
2-thread of $H$ to be sponsored by an incident $4$-vertex such that each
4-vertex sponsors at most one 2-thread.

If a 2-vertex $v$ has two 3-neighbors, call the 1-thread containing $v$ {\it
light}.  Let $J$ be the subgraph induced by light 1-threads.  We will show that
each component of $J$ must be a tree or a cycle.  Suppose instead that $J$
contains a cycle with an incident edge.  We denote the cycle by
$u_1v_1u_2v_2\ldots u_kv_k$ where $d(u_i)=2$ and $d(v_i)=3$ for all $i$ and
$v_1$ is adjacent to a 2-vertex $z$ not on the cycle (which is adjacent to a
second 3-vertex).  By minimality, we can
6-list-color $(G\setminus \{u_1,v_1,u_2,z\})^2$.  Now only three neighbors of
$v_1$ in $G^2$ are colored, so we can color $v_1$.  Finally, we uncolor each
vertex $u_i$.  Now the uncolored vertices induce in $G^2$ a subgraph $K$
consisting of a cycle with a single vertex $z$ adjacent to two successive
vertices on the cycle.  For each vertex $x\in V(K)$, let $L(x)$ denote the
colors available for $x$.  Note that we have $|L(x)|\ge d_K(x)$ for all $x\in
V(K)$.  Thus, by the Choosability Lemma, we can extend the list-coloring to all
of $V(G)$.  So each component of $J$ must be a tree or a cycle; hence we can
assign each 1-thread of $J$ to be sponsored by an incident 3-vertex such that
each 3-vertex sponsors at most one 1-thread.

We use the initial charge function $\mu(v)=d(v)$ and the following discharging
rules.

\begin{itemize}
\item {\bf R1:} Each $3$-vertex gives charge $1/9$ to each incident thread.
\item {\bf R2:} Each 4-vertex gives charge $3/9$ to each incident thread.
\item {\bf R3:} Each $3^+$-vertex incident with a sponsored thread
gives an additional charge of $2/9$ to that thread.
\end{itemize}

Now we show that each $3^+$-vertex finishes with charge at least $22/9$ and
that each $k$-thread receives charge at least $4k/9$ (so it finishes with
charge at least $22k/9$).  As in Lemma~\ref{lemma1}, note that $\delta(G)\ge 2$.
If $d(v)=4$, then $v$ gives charge $3/9$ to each incident thread and an
additional $2/9$ to at most one sponsored thread, so $\mu^*(v)\ge
4-4(3/9)-1(2/9)=22/9$.
If $d(v)=3$, then $v$ sends charge $1/9$ to each incident thread and an
additional $2/9$ to at most one incident thread, so $\mu^*(v)\ge 3 -
3(1/9)-1(2/9)=22/9$.

Now we consider threads.
Each 2-thread receives charge $3/9$ from each endpoint and charge 2/9 from its
sponsor, for a total charge of $8/9$.  Consider a 1-thread with interior
2-vertex $v$.  If $v$ has at least one 4-neighbor, then the 1-thread receives
charge at least $3/9+1/9=4/9$.  Each 1-thread with both endpoints of degree 3
receives charge $1/9$ from each endpoint and charge $2/9$ from its
sponsor for a total charge of $4/9$.  Thus $\mad(G)\ge 22/9$.  This
contradiction completes the proof.
\end{proof}

\begin{lem}
\label{lemma3}
If $\Delta(G)\le 4$ and $\mad(G)<18/7$, then $\chi_l(G^2)\le 7$.  In particular,
for every planar graph $G$ with $\Delta(G)\le 4$ and girth at least 9, we have
$\chi_l(G^2)\le 7$.
\end{lem}
\begin{proof}
The second statement follows from the first by Fact~1.  To prove the first,
we use discharging.  Let $G$ be a minimal counterexample to the lemma.
For each vertex $v$, we begin with charge $\mu(v)=d(v)$,
and we will show that after discharging each vertex finishes with charge at
least $18/7$, which gives a contradiction and proves the lemma.
We leave to the reader the details of verifying that each of the three
following configurations is 7-reducible (see Fig.~\ref{slika2}):
\begin{itemize}
    \item[$(i)$] \textit{A thread of two 2-vertices;}
    \item[$(ii)$] \textit{A 3-vertex adjacent to three 2-vertices;}
    \item[$(iii)$]\textit{A 3-vertex, adjacent to two 2-vertices, one of which 
        is adjacent to a second 3-vertex.}
\end{itemize}

We use the initial charge function $\mu(v)=d(v)$ and the following discharging
rules.

\begin{itemize}
\item {\bf R1:} Each $4$-vertex gives charge $5/14$ to each 2-neighbor.
\item {\bf R2:} Each $3$-vertex with a single 2-neighbor gives charge $4/14$ to
that 2-neighbor.
\item {\bf R3:} Each $3$-vertex with two 2-neighbors gives charge $3/14$ to
each 2-neighbor.
\end{itemize}

Now we show that each vertex finishes with charge at least $18/7$.
Note that $\delta(G)\ge 2$.
If $d(v)=2$ and $v$ has a 4-neighbor, then $\mu^*(v)\ge 2+ 5/14+3/14=2+4/7$.
If $d(v)=2$ and $v$ has no 4-neighbor, then $v$ receives charge $4/14$ from each
of its 3-neighbors, since otherwise we have configuration $(iii)$ in Figure~\ref{slika2}.
Now $\mu^*(v)\ge 2+2(4/14)=18/7$.
If $d(v)=3$, then by $(ii)$ $v$ has at most two 2-neighbors, so $\mu^*(v)\ge
3-2(3/14)=18/7$.  Finally, if $d(v)=4$, then $v$ has at most four 2-neighbors,
so $\mu^*(v)\ge 4 - 4(5/14)=18/7$.  Thus, $\mad(G)\ge 18/7$.  This
contradiction completes the proof.
\end{proof}

\begin{figure}[htp!]
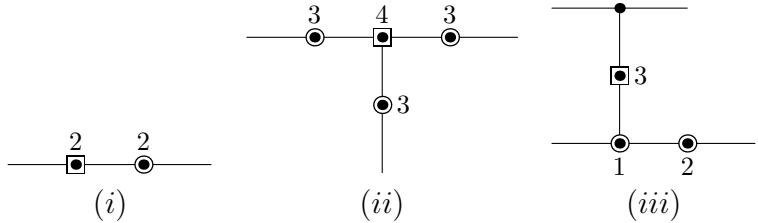

	\centerline{
		\begin{tabular}{ccc}
            \includegraphics[scale=1]{Fig3.1}\hspace{1cm}&
            \includegraphics[scale=1]{Fig3.2}\hspace{1cm}&
            \includegraphics[scale=1]{Fig3.3}\\
            $(i)$ & $(ii)$&$(iii)$
		\end{tabular}}
\caption{Configurations $(i)$, $(ii)$, and $(iii)$ from Lemma~\ref{lemma3} are 7-reducible. }
\label{slika2}
\end{figure}

\begin{lem}
\label{lemma4}
If $\Delta(G)\le 4$ and $\mad(G)<14/5$, then $\chi_l(G^2)\le 8$.  In particular,
for every planar graph $G$ with $\Delta(G)\le 4$ and girth at least 7, we have
$\chi_l(G^2)\le 8$.
\end{lem}
\begin{proof}
The second statement follows from the first by Fact~1.  To prove the first,
we use discharging.  Let $G$ be a minimal counterexample to the lemma.
For each vertex $v$, we begin with charge $\mu(v)=d(v)$,
and we will show that after discharging each vertex finishes with charge at
least $14/5$, which gives a contradiction and proves the lemma.
We call a 2-vertex with two 3-neighbors a {\it light 2-vertex}.
We call a 2-vertex with a 3-neighbor and a 4-neighbor a {\it medium 2-vertex}.
We call a 2-vertex with two 4-neighbors a {\it heavy 2-vertex}.
We call a 3-vertex adjacent to a light 2-vertex a {\it needy 3-vertex}.
Below we note that adjacent 2-vertices are 8-reducible.  This implies that every
2-vertex is heavy, medium, or light.

We leave to the reader the details of verifying the following 8-reducible
configurations.
Recall from the previous lemma that adjacent 2-vertices are 7-reducible, so they
are also 8-reducible. We will use the folowing 8-reducible configurations (see Fig.~\ref{slika22}):
\begin{itemize}
    \item[$(i)$] \textit{ a 3-vertex with two 2-neighbors;}
    \item[$(ii)$] \textit{ a 3-vertex with two 3-neighbors and a light 2-neighbor;}
    \item[$(iii)$] \textit{ a 4-vertex with three 2-neighbors, one of which is medium;}
    \item[$(iv)$] \textit{A 4-vertex with a needy 3-neighbor and two 2-neighbors, one of which is medium.}
\end{itemize}
If a 1-thread $S$ contains a heavy 2-vertex $v$ then we call $S$ {\it heavy}.
Let $J$ be the subgraph induced by heavy 1-threads.  
Each component of $J$ must be a tree or a cycle.  Since the proof is identical
to that given in Lemma~\ref{lemma2}, here we do not repeat the details.  Since
each component of $J$ is a tree or a cycle, we can assign each 2-vertex on a
heavy 1-thread to be sponsored by an adjacent 4-vertex, so that each 4-vertex
sponsors at most one such 2-vertex.

We use the initial charge function $\mu(v)=d(v)$ and the following discharging
rules.

\begin{itemize}
\item {\bf R1:} Each $3^+$-vertex gives charge $1/5$ to each adjacent 2-vertex.
\item {\bf R2:} Each 4-vertex gives charge $1/5$ to each adjacent needy 3-vertex.
\item {\bf R3:} Each needy 3-vertex gives an additional $1/5$ to each adjacent light
2-vertex.
\item {\bf R4:} Each 4-vertex gives an additional $2/5$ to each adjacent medium
2-vertex and each adjacent sponsored 2-vertex.
\end{itemize}

\begin{figure}[htp!]
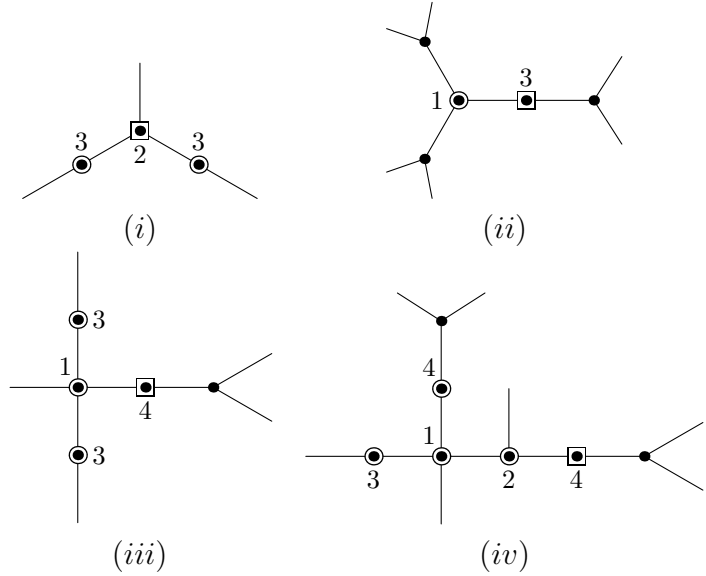

	\centerline{
		\begin{tabular}{cc}
            \includegraphics[scale=1]{Fig4.1}\hspace{1cm}&
            \includegraphics[scale=1]{Fig4.2}\\
            $(i)$ & $(ii)$\\
            \includegraphics[scale=1]{Fig4.3}\hspace{1cm}&
            \includegraphics[scale=1]{Fig4.4}\\
            $(iii)$ & $(iv)$
		\end{tabular}}
\caption{Configurations $(i)$--$(iv)$ from Lemma~\ref{lemma4} are 8-reducible.}
\label{slika22}
\end{figure}

Now we show that each vertex finishes with charge at least $14/5$.  Note that
$\delta(G)\ge 2$.

Suppose $d(v)=2$.
If $v$ is heavy, then $v$ receives charge $1/5$ from each neighbor and an
additional charge $2/5$ from its sponsor, so $\mu^*(v)=2+2(1/5)+2/5=14/5$.
If $v$ is medium, then $v$ receives charge $1/5$ from its 3-neighbor and charge
$1/5+2/5$ from its 4-neighbor, so $\mu^*(v)=2+1/5+1/5+2/5=14/5$.
If $v$ is light, then $v$ receives charge $1/5$ from each neighbor and an
additional charge $1/5$ from each neighbor, so $\mu^*(v)=2+2(2/5)=14/5$.

Suppose $d(v)=3$.  By $(i)$, $v$ has at most one 2-neighbor.  If $v$ has a light
2-neighbor, then $v$ gives it charge $1/5+1/5$ and $v$ receives charge $1/5$
from some 4-neighbor, since otherwise we have configuration $(ii)$.  So
$\mu^*(v)\ge 3-2/5+1/5=14/5$.  If $v$ has a medium 2-neighbor, then $v$ gives
it only charge $1/5$, so $\mu^*(v)\ge 3-1/5=14/5$.

Suppose $d(v)=4$.  If $v$ has no medium neighbors, then $v$ gives charge at most
$1/5$ to each neighbor and an additional charge of $2/5$ to at most one
sponsored 2-vertex, so $\mu^*(v)\ge 4-4(1/5)-2/5=14/5$.  So suppose that $v$
has a medium 2-neighbor.  If $v$ has only one 2-neighbor, then $v$ gives charge
at most 1/5 to each other neighbor and charge $1/5+2/5$ to its medium
2-neighbor, so $\mu^*(v)\ge 4-3(1/5)-1/5-2/5=14/5$.  If $v$ has at least two
2-neighbors, at least one of which is medium, then by configurations $(iii)$ and
$(iv)$, $v$ gives charge to no neighbors besides these two 2-neighbors.  Since
$v$ gives total charge at most $3/5$ to each of these 2-neighbors,
$\mu^*(v)\ge 4-2(3/5)=14/5$.

Thus, each vertex finishes with charge at least $14/5$, so $\mad(G)\ge
14/5$.  This contradiction completes the proof.
\end{proof}

\begin{lem}
\label{lemma5}
If $\Delta(G)\le 4$ and $\mad(G)<10/3$, then $\chi_l(G^2)\le 12$.  In
particular, for every planar graph $G$ with $\Delta(G)\le 4$ and girth at least
5, we have $\chi_l(G^2)\le 12$.
\end{lem}
\begin{proof}
The second statement follows from the first by Fact~1.  To prove the first,
we use discharging.  Let $G$ be a minimal counterexample to the lemma.
For each vertex $v$, we begin with charge $\mu(v)=d(v)$,
and we will show that after discharging each vertex finishes with charge at
least $10/3$, which gives a contradiction and proves the lemma.

\begin{figure}[htp!]
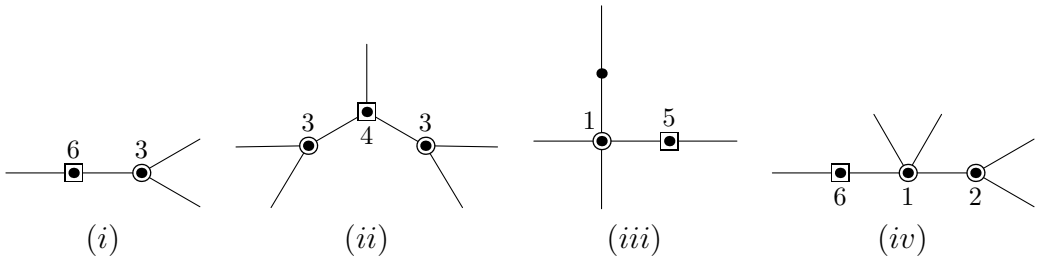

	\centerline{
		\begin{tabular}{cccc}
            \includegraphics[scale=1]{Fig5.1}\hspace{.5cm}&
            \includegraphics[scale=1]{Fig5.2}\hspace{.5cm}&
            \includegraphics[scale=1]{Fig5.3}\hspace{.5cm}&
            \includegraphics[scale=1]{Fig5.4}\\
            $(i)$ & $(ii)$& $(iii)$ & $(iv)$
		\end{tabular}}
\caption{Configurations $(i)$--$(iv)$ from Lemma~\ref{lemma5} are
12-reducible.}
\label{slikaQ15}
\end{figure}

We leave to the reader the details of verifying 12-reducibility of the following four 
configurations (see Fig.~\ref{slikaQ15}):
\begin{itemize}
    \item[$(i)$] \textit{a 2-vertex adjacent to a 3-vertex;}
    \item[$(ii)$] \textit{a 3-vertex adjacent to two 3-vertices;}
    \item[$(iii)$] \textit{a 4-vertex with two adjacent 2-vertices;}
    \item[$(iv)$] \textit{a 4-vertex adjacent to a 2-vertex and a 3-vertex.}
\end{itemize}
We use the initial charge function $\mu(v)=d(v)$ and the following discharging
rules.

\begin{itemize}
\item {\bf R1:} Each 4-vertex gives charge $2/3$ to each adjacent 2-vertex.
\item {\bf R2:} Each 4-vertex gives charge $1/6$ to each adjacent 3-vertex.
\end{itemize}

Now we show that each vertex finishes with charge at least $10/3$.
Note that $\delta(G)\ge 2$.
If $d(v)=2$, then by $(i)$ both neighbors of $v$ are 4-vertices, so
$\mu^*(v)=2+2(2/3)=10/3$.
If $d(v)=3$, then by $(i)$ $v$ has no 2-neighbors and by $(ii)$ $v$ has two
4-neighbors, so $\mu^*(v)\ge3+2(1/6)=10/3$.
Suppose that $d(v)=4$.  If $v$ has a 2-neighbor, then by $(iii)$ and $(iv)$ $v$ has
no other $3^-$-neighbor, so $\mu^*(v)\ge 4-2/3=10/3$.  If $v$ has no
2-neighbor, then $\mu^*(v)\ge 4-4(1/6)=10/3$.
Thus, each vertex finishes with charge at least $10/3$.  This contradiction
completes the proof.
\end{proof}

\begin{lem}
\label{lemma6}
If $G$ is planar and $\Delta(G)\le 4$, then $\chi_l(G^2)\le 14$.
\end{lem}
\begin{proof}
Let $G$ be a minimal planar graph with $\chi_l(G^2)>14$.
The following six configurations are 14-reducible (see Fig.~\ref{slika6}):
\begin{itemize}
    \item[$(i)$] \textit{a 2-vertex;}
    \item[$(ii)$] \textit{two adjacent 3-vertices;}
    \item[$(iii)$] \textit{a 3-vertex incident to a 3-face;}
    \item[$(iv)$] \textit{a 3-vertex incident to a 4-face;}
    \item[$(v)$] \textit{a 4-vertex incident to two 3-faces (sharing an edge or not);}
    \item[$(vi)$] \textit{a 4-vertex incident to a 3-face and a 4-face (sharing an edge or not).} 
\end{itemize}

We use discharging with the following initial charges:
\begin{itemize}
\item $\mu(v)=2d(v)-6$ for each vertex $v$.
\item $\mu(f)=\l(f)-6$ for each face $f$.
\end{itemize}
By Euler's formula, the sum of the charges is negative.
We use the following discharging rules.
\begin{itemize}
\item {\bf R1:} Each 4-vertex gives charge 1 to each incident 3-face.
\item {\bf R2:} Each 4-vertex gives charge 1/2 to each incident 4-face.
\item {\bf R3:} Each 4-vertex gives charge 1/3 to each incident 5-face.
\end{itemize}
\begin{figure}[htp!]
	\centerline{
    \begin{tabular}{cc}
		\begin{tabular}{cc}
            \includegraphics[scale=1]{fig61.1}\hspace{.5cm}&
            \includegraphics[scale=1]{fig61.2}\\
            {}\\
            $(i)$ & $(ii)$\\
            {}\\
            \includegraphics[scale=1]{fig61.5}&
            \includegraphics[scale=1]{fig61.6}
        \end{tabular}&
        \begin{tabular}{cc}
            \includegraphics[scale=1]{fig61.3}\hspace{.5cm}&
            \includegraphics[scale=1]{fig61.4}\\
            {}\\
            $(iii)$ & $(iv)$\\
            {}\\
            \includegraphics[scale=1]{fig61.7}&
            \includegraphics[scale=1]{fig61.8}
        \end{tabular}\\
         $(v)$ & $(vi)$\\
    \end{tabular}}
\caption{Configurations $(i)$--$(vi)$ from Lemma~\ref{lemma6} are 14-reducible.}
\label{slika6}
\end{figure}

Now we show that all vertices and faces finish with nonnegative charge, which
is a contradiction.
By $(i)$, we have $\delta(G)\ge 3$.  Thus, we must verify that each $5^-$-face
receives sufficient charge and that no 4-vertex gives away too much charge; note
that $4$-vertices give charge only to faces.

If $\l(f)=3$, then by $(iii)$ each incident vertex is a 4-vertex, so
$\mu^*(f)=-3+3(1)=0$.  If $\l(f)=4$, then by $(iv)$ each incident vertex is a
4-vertex, so $\mu^*(f)=-2+4(1/2)=0$.  If $\l(f)=5$, then (since $G$ has no
adjacent 3-vertices by $(ii)$), $f$ has at least three incident 4-vertices, so
$\mu^*(f)\ge -1 + 3(1/3)=0$.

If $d(v)=3$, then $\mu^*(v)=\mu(v)=0$.  If $d(v)=4$ and $v$ is incident to a
triangle, then by $(v)$ and $(vi)$ vertex $v$ is also incident to three
$5^+$-faces, so $\mu^*(v)\ge 2-1-3(1/3)=0$.  If $d(v)=4$ and $v$ is not
incident to a triangle, then $\mu^*(v)\ge 2-4(1/2)=0$.

Thus, each vertex and face finishes with nonnegative charge.  This contradicts
the fact that the sum of the initial charges was negative.  This contradiction
completes the proof.
\end{proof}

\end{document}